\documentclass[twoside,12pt]{article}
\date{\today}
\usepackage{amsmath,amsthm,amsfonts,latexsym,amssymb,fancybox,comment,graphicx,subfig,array,subeqnarray,indentfirst}
\usepackage{hyperref,mathrsfs}
\usepackage[mathscr]{euscript}
\usepackage{multirow,float}
\restylefloat{table}
\usepackage{colortbl}
\def\cor#1{{\color{black}#1}}

\date{}

\title{Simple quadrature rules for a nonparametric nonconforming quadrilateral elements
}\author{Kanghun Cho\thanks{
 Samsung Fire \& Marine Insurance Co., Ltd., 14, Seocho-daero 74-gil, Seocho-gu, Seoul 06620,
         Korea;
E-mail: serein@snu.ac.kr
}~~~ Dongwoo Sheen\thanks{Department of Mathematics,
         Seoul National University,
         Seoul 08826,
         Korea
E-mail: sheen@snu.ac.kr
}}

\numberwithin{equation}{section}

\def\and{\quad\text{and}\quad}

\def\barx{\bar{x}}

\def\bbb{\mathbf{\bar{b}}}

\def\bbe{\bar{\mathbf{e}}}
\def\bbg{\bar{\mathbf{g}}}

\def\bbm{\bar{\mathbf{m}}}
\def\bbx{\bar{\mathbf x}}
\def\bcA{\bar{\mathcal A}}

\def\bd{\mathbf d}

\def\bh{\bar{h}}

\def\bkap{\boldsymbol{\kappa}}

\def\bm{\mathbf m}

\def\bv{\mathbf v}
\def\bv{\mathbf{ v}}
\def\bv{\mathbf{v}}

\def\bx{\bar{x}}
\def\bxi{\boldsymbol \xi}
\def\bx{\mathbf x}
\def\bx{\mathbf{x}}
\def\cA{\mathcal{A}}

\def\cE{\mathcal{E}}
\def\cF{\mathcal{F}}
\def\cQ{\mathcal Q}

\def\cS{\mathcal{S}}

\def\dbh{\bar{\bar{h}}}
\def\dbbh{\bar{\bar{\mathbf h}}}

\def\den{\operatorname{D}}

\def\div{\nabla\cdot\,}
\def\div{\nabla\cdot}

\def\dx{\,\operatorname{d}\mathbf x}

\def\foreach{\,\forall\,}
\def\forany{\,\forall\,}

\def\grad{\nabla\,}

\def\hbx{\hat{\mathbf{x}}}
\def\hbx{\widehat{\mathbf{x}}}

\def\hx{\hat{x}}

\newcommand{\jump}[2]{\left[\left[#1\right]\right]_{#2}}
\def\kap{\kappa}

\def\<{\left\langle}

\def\mbm{\mathbf m}
\def\mbR{\mathbb R}
\def\mbx{\mathbf x}
\def\mbx{\mathbf{x}}

\def\mcF{\mathcal F}
\def\NC{\mathcal{NC}}
\def\num{\operatorname{N}}

\def\O{\Omega}
\def\p{\partial}
\def\qq{\qquad}
\def\q{\quad}

\def\>{\right\rangle}
\def\sig{\sigma}

\def\Span{\operatorname{Span}}
\def\sq{\mathscr{Q}}

\def\suchthat{\,\mid\,}

\def\Tau{\mathcal{T}}

\def\tbs{\tilde{\mathbf{s}}}

\def\tbx{\tilde{\mathbf{x}}}

\def\til{\widetilde}

\def\ts{\widetilde s}

\def\tx{\tilde{x}}
\def\tx{\widetilde{x}}

\def\cV{\mathcal V}

\newcommand{\bal}{\begin{aligned}}
\newcommand{\be}{\begin{eqnarray}}
\newcommand{\bes}{\begin{eqnarray*}}
\newcommand{\bs}{\begin{subeqnarray}}
\newcommand{\bss}{\begin{subeqnarray*}}

\newcommand{\eal}{\end{aligned}}
\newcommand{\ee}{\end{eqnarray}}
\newcommand{\ees}{\end{eqnarray*}}
\newcommand{\eq}[1]{\begin{eqnarray}\label{#1}}
\newcommand{\es}{\end{subeqnarray}}
\newcommand{\ess}{\end{subeqnarray*}}

\newcommand{\nn}{\nonumber}

\newcommand{\qe}{\end{eqnarray}}

\newtheorem{theorem}{Theorem}[section]

\newtheorem{definition}[theorem]{Definition}
\newtheorem{example}[theorem]{Example}
\newtheorem{lemma}[theorem]{Lemma}
\newtheorem{proposition}[theorem]{Proposition}
\newtheorem{remark}[theorem]{Remark}
\renewcommand{\hat}{\widehat}
\renewcommand{\tilde}{\widetilde}

\begin{document}
\maketitle
\begin{abstract}
  We introduce simple quadrature rules for the
  family of nonparametric nonconforming quadrilateral element with
  four degrees of freedom. Our quadrature rules are motivated by the
  work of Meng {\it et al.} \cite{meng2018new}. First, we introduce a
  family of MVP (Mean Value Property)-preserving four DOFs
  nonconforming elements
  on the intermediate reference domain introduced by Meng {\it et
    al.}. Then we design two--points and three--points quadrature
  rules on the intermediate reference domain.
  Under the assumption on equal quadrature weights, the deviation from the
  quadrilateral center of the Gauss points for the two points and
  three points rules assumes the same quadratic polynomials with
  constant terms modified. Thus, the two--points rule and three--points
  rule are constructed at one stroke. The quadrature rules are
  asymptotically optimal as the mesh size is sufficiently small.
  Several numerical experiments are carried out, which show efficiency
  and convergence properties of the new quadrature rules.
\end{abstract}


\section{Introduction}
Sevearal effects of numerical integration have been studied in various
aspects. In particular, the Gaussian quadrature rules for general
domains goes back to the monograph by
Stroud and Secrest \cite{stroud1966gaussian}, and
Herbold {\it et al.} \cite{herbold1969effect,
  herbold1971effect} 
reported extensive analysis for the effects of numerical integration
of variational equations.
Ciarlet and Raviart \cite{ciarlet1972combined} investigated the
numerical effects in finite element methods and Ciarlet
\cite{ciarlet1991finite} describes in detail the numerical quadrature effects in
approximating finite element methods.
More recently, such effects were extensively investigated
in the $p$--version finite elements by Banerjee and Suri
\cite{banerjee1992effect}, and in the approximation of eigenvalues
by Banerjee and Osborn \cite{banerjee1989estimation, osborn1990estimation}. More recently
Banergee and Babu{\v{s}}ka {\it et al.}\cite{banerjee1989estimation, babuska2011effect} studied such effects in the approximation
of linear functionals including eigenvalue approximations.

In the meanwhile, it has been
well-known that the lowest degree conforming finite element pairs
lead to unstable numerical solutions in the numerical
simulation of fluid and solid mechanics. 
A proper choice of nonconforming finite element spaces in the
approximation of vector variables heals such
kind of instability \cite{crouzeix-raviart, fortin-soulie-quad-nc-2d,
  fortin3d, han84, linke2016robust, rannacher-turek, turek, cdy, cdssy} in the
approximation of incompressible fluid flows. Contrary to the
simplicial nonconforming elements, most quadrilateral nonconforming
elements contain extra polynomials to $P_k$
\cite{achchab2014simple, achchab2018new, arbogast1995implementation,
  dssy-nc-ell, han84, linke2016robust, rannacher-turek, turek, zhou2016new} which require
additional quadrature points, although there are some
quadrilateral elements consisting of $P_k$ only
\cite{park-sheen-p1quad, altmann-carstensen, zeng2020optimal}.
In our paper, we limit our interest to 4-DOFs
quadrilateral nonconforming elements of lowest order.
In \cite{rannacher-turek}, Rannacher and Turek introduced the rotated $Q_1$
nonconforming elements consisting of $P_1(\hat K)\oplus
\Span\{\hx_1^2-\hx_2^2\}$
on the reference domain $\hat K=[-1,1]^2$
with two types of degrees of freedom:
(1) the four mean edge integral DOFs and 
(2) the four edge-midpoint value DOFs. The two types of DOFs lead to
different numerical results.  The use of
edge-midpoint values is cheaper and simpler than that of
mean edge integral values in calculating the basis functions.
Douglas {\it et al.} modified the Rannacher--Turek element by
replacing quadratic polynomial to a quartic polynomial, where
the two types of DOFs are identical on rectangular meshes
\cite{dssy-nc-ell}. We will call the element \cite{dssy-nc-ell} as the
DSSY element, which fulfills the property $\frac1{|e|}\int_e \phi\,ds =
\phi(m),$ for every edge $e$ and its midpoint $m,$ which will be 
coined as MVP (Mean value Property) throughout the paper.
For truly quadrilaterals, a class of nonparametric
DSSY element \cite{jeon2013class} was introduced.

Recently, an interesting observation was made for quadrature rules for
nonconforming quadrilateral element by Meng, Cui, and Luo (hereafter,
abbreviated by MCL), and a new type of nonconforming element was
introduced in \cite{meng2018new}. The MCL element consists of
$P_1(\bar K)\oplus \Span\{\barx_1\barx_2\}$ on each MCL-type quadrilateral
$\bar K,$ which will be explained in the following section. Then a
simple effective quadrature rules with three points in $\bar K$ is
defined.
In \cite{meng2018new}, the basis functions are at most of degree two
so that the quadrature formulae of degree two is found.
However the class of DSSY finite elements contains high-order degree polynomial
bases to fulfill MVP, and thus the quadrature formula in
\cite{meng2018new} does not guarantee optimal convergence.

Our aim in this paper is to investigate whether it is possible
to define similar two-point and 
three--point quadrature rules for the class of DSSY elements. We
construct a class of nonparametric DSSY element on MCL-type
quadrilaterals. It
turns out to be possible to find a two-point rule and a three points
rules of precision 1 at one stroke under the assumption on the equal weights and
geometrically symmetric points with respect to barycenters.
We show
optimal convergence in broken energy norm under the condition that the
mesh sizes are sufficiently small.

The organization of the paper is as follows. In the next section, we
expose some notations and preliminaries, and then
briefly review some quadrilateral nonconforming elements which have
four DOFs. In Section \ref{sec:npDSSY}, we introduce a class of
nonparametric DSSY elements in quadrilaterals of the type
used by Meng {\it et al.} \cite{meng2018new}. Then in Section
\ref{sec:effects} the effects of numerical integration are analyzed.
Then we construct two-point and three-point quadrature rules
in Section \ref{sec:const-formula}. In the final section we provide
some numerical results which confirm the theories developed so far.

\section{Quadrilateral nonconforming elements}\label{sec:review-nonconform}
\subsection{Notations and Preliminaries}
Let $\Omega$ be a simply-connected polygonal domain in $\mathbb{R}^2$
and $(\mathcal{T}_h)_{h>0}$ be a family of shape regular
convex quadrilateral triangulations of $\Omega$ with
$\text{max}_{K \in T_h} \text{diam}(K)=h.$
If $D$ is a polygonal domain or a triangulation, 
denote by $\cE(D), \cE^i(D),\cE^b(D)$ the set of all edges, interior
edges,
and boundary edges, respectively, in $D;$ also
by $\cV(D), \cV^i(D),\cV^b(D)$ the set of all vertices, interior
vertices,
and boundary vertices, respectively, in $D.$ If $D=\Tau_h,$ the
notations will be simplified to $\cE_h,\cE_h^i,\cE_h^b,\cV_h,
\cV_h^i,\cV_h^b,$ and so on. For edge $e,$ denote by $\mbm_e$ the midpoint of $e.$

For a typical element $K\in \Tau_h,$
denote $\bv_j, j=1,\cdots,4,$ the four vertices of $K.$
Also denote
by $e_j$ the edge passing through $\bv_{j-1}$ and $\bv_{j}$
and by $\mathbf{m}_j$
the midpoint of $e_j$ for $j=1,\cdots,4,$ (with identification of
indices by $\mod(4)$ such as $\bv_0:=\bv_4,$ etc.)
Let $\hat K=[-1,1]^2$ be the reference cube with
$\hat\bv_1={1\choose 1},\hat\bv_2={-1\choose 1},
\hat\bv_3={-1\choose -1},\hat\bv_4={1\choose -1},$ with the four
midpoints
$\hat\bm_1={1\choose 0},\hat\bm_2={0\choose 1},
\hat\bm_3={-1\choose 0},\hat\bm_4={0\choose -1}.$
Define the linear functionals
$\sig_e^{(k)}\in (C^0(\overline D))',k=i,m,$
  for all $e\in \cE(D) $ by
\begin{eqnarray}
  \sig^{(i)}_{e}(v)=\frac1{|e|}\int_{e} v\dx;\q
  \sig^{(m)}_e(v)=v(\mbm_e)\forany v\in C^0(\bar D).
\end{eqnarray}
We also denote by $\jump{f}{e}$ the jump of $f$ across edge $e$ such
that $\jump{f}{e} = (f_k-f_j)|_e$ where $f_j$ and $f_k$ denote the
restrictions of $f$ to $K_j$ and $K_k$ where $e=K_j\cap K_k.$ If $e\in
\cE_h^b,$  $\jump{f}{e} = -f|_e.$

Let $\mcF_K$ denote an invertible bilinear map which maps $\hat K$ onto $K.$

For any open subset $\Omega$ of $\mathbb{R}^n$, denote the seminorm
and norm of the Sobolev space $W^{k,p}(\Omega)$ by
$|\cdot|_{k,p,\Omega}$ and $||\cdot||_{k,p,\Omega}$,
respectively. Also denote by $H^k(\Omega) = W^{k,2}(\Omega)$ and
abbreviate $|\cdot|_{k,p,\Omega}$ and $||\cdot||_{k,p,\Omega}$ as
$|\cdot|_{k,\Omega}$ and $||\cdot||_{k,\Omega}$.
Define the broken norms and seminorms on broken Sobolev spaces as
follows.
\begin{eqnarray*}
  |v_h|_{k,p,h} &=& \begin{cases}
    (\sum_{K\in\Tau_h} |v_h|_{k,p,K}^p)^{\frac1{p}}, &\, p \in [1,\infty),\\
    \max_{K\in\Tau_h} |v_h|_{k,\infty,K} &\, p = \infty,
  \end{cases}\quad
\|v_h\|_{k,p,h} = \begin{cases} \left[\sum_{j=0}^k
    |v_h|_{k,p,h}^p\right]^{\frac1{p}}, &\, p \in [1,\infty),\\
  \max_{j=0}^k |v_h|_{j,\infty,h} &\, p = \infty,
\end{cases}\\
  W^{k,p}(\Tau_h) &=& \{v_h\in L^p(\O)\,|\, \|v\|_{k,p,h} < \infty\}.
\end{eqnarray*}
If $p=2,$ the subindices $p$ can be omitted as usual.

\subsection{The Rannacher--Turek and DSSY nonconforming elements}
The parametric nonconforming quadrilateral elements are
defined as follows:
\begin{definition}[Parametric nonconforming
  quadrilateral finite element] \label{def:rannacher-turek}
For $k=i,m,$  define the reference element
$(\hat{K},\hat{P}^{pNC}_{\hat{K}},
\hat{\Sigma}^{pNC,(k)}_{\hat{K}})$ by
\begin{enumerate}
\item $\hat{K}=[-1,1]^2;$
\item $\hat{P}^{pNC}_{\hat{K}} =\begin{cases} P_1(\hat K)\oplus
    \Span\{\hat{x}_1^2-\hat{x}_2^2 \}, &\text{ if } NC=RT,\\
    P_1(\hat K)\oplus \Span\{\hat{x}_1^2-\hat{x}_2^2 - \frac53
    (\hx_1^4 -\hx_2^4)\},&\text{ if } NC=DSSY;
  \end{cases}$
\item$\hat{\Sigma}^{pNC,(k)}_{\hat{K}}=\{\sig_e^{(k)}\foreach e
  \in \cE(\hat K)\},$ for NC=RT or NC=DSSY.
\end{enumerate}

For $k=i,m,$ 
the global parametric Rannacher-Turek finite element spaces
\cite{rannacher-turek} and DSSY finite element spaces \cite{dssy-nc-ell}, with
NC=RT and NC=DSSY, respectively,
are defined by
\begin{equation*}
\begin{split}
  &\NC^{pNC,(k)}_{h}=\{ v \in L^2(\O) \suchthat v|_K = \hat v\circ
  F_K^{-1} \text{ for some } \hat v \in \hat P^{pNC,(k)}_{\hat K}
  \foreach K \in \Tau_h, \\
  &\q\q\q\q\q\q\q\q\q\q\q \sig_e^{(k)}(\jump{v}{e})=0 \foreach e \in \cE^i_h \},\\
&\NC^{pNC,(k)}_{h,0}=\{v \in \NC^{pNC,(k)}_{h} \suchthat \sig_e^{(k)}(v) =0 \foreach
~ e \in \cE_h^b\}.
\end{split}
\end{equation*}
\end{definition}
Notice that the main additional feature of the DSSY element is the
MVP:
\begin{equation}\label{eq:mvp}
  \sig^{(i)}(\phi) = \sig^{(m)}(\phi) \forany e\in \cE_h
  \forany \phi\in \NC^{pDSSY}_h.
\end{equation}
The nonparametric elements were introduced as follows.
\begin{definition}[Nonparametric Rannacher-Turek nonconforming quadrilateral
finite element] \label{def:rannacher-turek-nonpara}
The {\it nonparametric} nonconforming Rannacher-Turek element $(K,P_K,\Sigma_K)$
on $K$ is defined by
\begin{enumerate}
\item $K$ is a convex quadrilateral;
\item $P_K^{npRT} = P_1(K)\oplus\Span\{\xi_1^2 - \xi_2^2\},\,
  \xi_j,j=1,2,$ are the two coordinates connecting the two pairs
    of opposite edge-midpoints of $K;$
\item ${\Sigma}^{npRT}_{K}=\{\sig_{e}^{(i)}\foreach e \in \cE(K)\}.$
\end{enumerate}
  The global nonparametric Rannacher-Turek finite element space is defined by
\begin{equation*}
\begin{split}
  &\NC^{npRT}_{h}=\{ v \in L^2(\O) \suchthat v|_K \in P^{npRT}_{K}
 \foreach K \in \Tau_h,  \sig_e^{(i)}(\jump{v}{e})=0 \foreach e
 \in \cE^i_h\},\\
&\NC^{npRT}_{h,0}=\{v \in \NC^{npRT}_{h} \suchthat \sig_e^{i}(v) =0 \foreach
~ e \in \cE_h^b\}.
\end{split}
\end{equation*}
\end{definition}
For the definition of nonparametric DSSY element, we need to introduce
an intermediate quadrilateral, denoted by $\til K,$ 
which is the image of $\hat K$ under the following simple bilinear map
$$
S\hbx = \hbx + \hx_1\hx_2 \tbs
$$ where $\tbs=A_K^{-1}\bd,
A_K = \frac14\left(\bv_1-\bv_2-\bv_3+\bv_4,
  \bv_1+\bv_2-\bv_3-\bv_4\right),
\bd = \frac14(\bv_1-\bv_2+\bv_3-\bv_4).
$ Then, the bilinear map $\mathcal F_K$ mapping $\hat K$ to $K$
is represented by $\mathcal F_K(\hbx)=A_K(S_K(\hbx))+\mathbf b,$
$\mathbf b = \frac14(\bv_1+\bv_2+\bv_3+\bv_4).$ Then there exists an
invertible affine map $\til A_K$ which maps $\til K$ onto $K$ such
that $\cF_K = \til A_K\circ S_K.$
For these formula, see \cite[(2.5)]{jeon2013class}.
For the nonparametric DSSY element \cor{for arbitrary $\til
  c\in\mathbb R$}, set 
\begin{eqnarray}\label{eq:til-mu}
\tilde{\mu}(\tbx;\tilde{c})
=\til\ell_1(\tbx)\til\ell_2(\tbx)\til\cQ(\tbx;\til c),
\end{eqnarray}
where \cor{(\cite[(2.12)]{jeon2013class})}
\begin{subeqnarray*}
\til\ell_1(\tbx)&=&\tx_1-\tx_2-\ts_1+\ts_2;
\q\til\ell_2(\tbx)=\tx_1+\tx_2+\ts_1+\ts_2,\q \cor{\til r=\frac{\sqrt{6}}{5}\sqrt{\frac52-\ts_1^2-\ts_2^2}},\\
\tilde{\cQ}(\tbx;\tilde{c})&=&\left(\tx_1+\frac25\ts_2\right)^2+
\left(\tx_2+\frac25\ts_1\right)^2 - \til r^2 + \til c\left[
  (\tx_1+\frac25\ts_2)
  (\tx_2+\frac25\ts_1) + \frac6{25}\ts_1\ts_2\right].
\end{subeqnarray*}
Then the nonparametric DSSY element is defined as follows.

\begin{definition}[Nonparametric DSSY nonconforming quadrilateral
finite element] \label{def:dssy-nonpara}
The {\it nonparametric} nonconforming DSSY element $(K,P_K,\Sigma_K)$
on $K$ is defined by
\begin{enumerate}
\item $\til K=S_K(\hat K);$ 
\item $\til P_{\til K}^{npDSSY} = P_1(\til K)\oplus\Span\{\til\mu\},$
\item ${\til\Sigma}^{npDSSY}_{\til K}=\{\sig_{e}^{(i)}\foreach e
  \in \cE(\til K)\}= \{\sig_{e}^{(m)}\foreach e
  \in \cE(\til K)\}.$
\end{enumerate}
  The global nonparametric DSSY finite element space is defined by
\begin{equation*}
\begin{split}
  &\NC^{npDSSY}_{h}=\{ v \in L^2(\O) \suchthat v|_K = (\til v\circ
  \til A_K^{-1}) \text{ for some } \til v\in \til P^{npDSSY}_{\til K}
  \foreach K \in \Tau_h,\\
  &\hspace{8cm}\sig_e^{(i)}(\jump{v}{e})=0 \foreach e
 \in \cE^i_h\},\\
&\NC^{npDSSY}_{h,0}=\{v \in \NC^{npDSSY}_{h} \suchthat \sig_e^{i}(v|_{e}) =0 \foreach
~ e \in \cE_h^b \}.
\end{split}
\end{equation*}
\end{definition}
\subsection{The MCL nonconforming element}
Recently Meng {\it et al.} \cite{meng2018new} defined a nonparametric
quadrilateral element slightly differently from the above
nonparametric Rannacher-Turek element by using an
intermediate quadrilateral $\bar K.$
In order to briefly explain the notion of the intermediate
quadrilateral $\bar K$ of MCL type, the introduction of the
following line equations are useful.
\begin{figure}
 \centering
\includegraphics[width=0.50\textwidth]{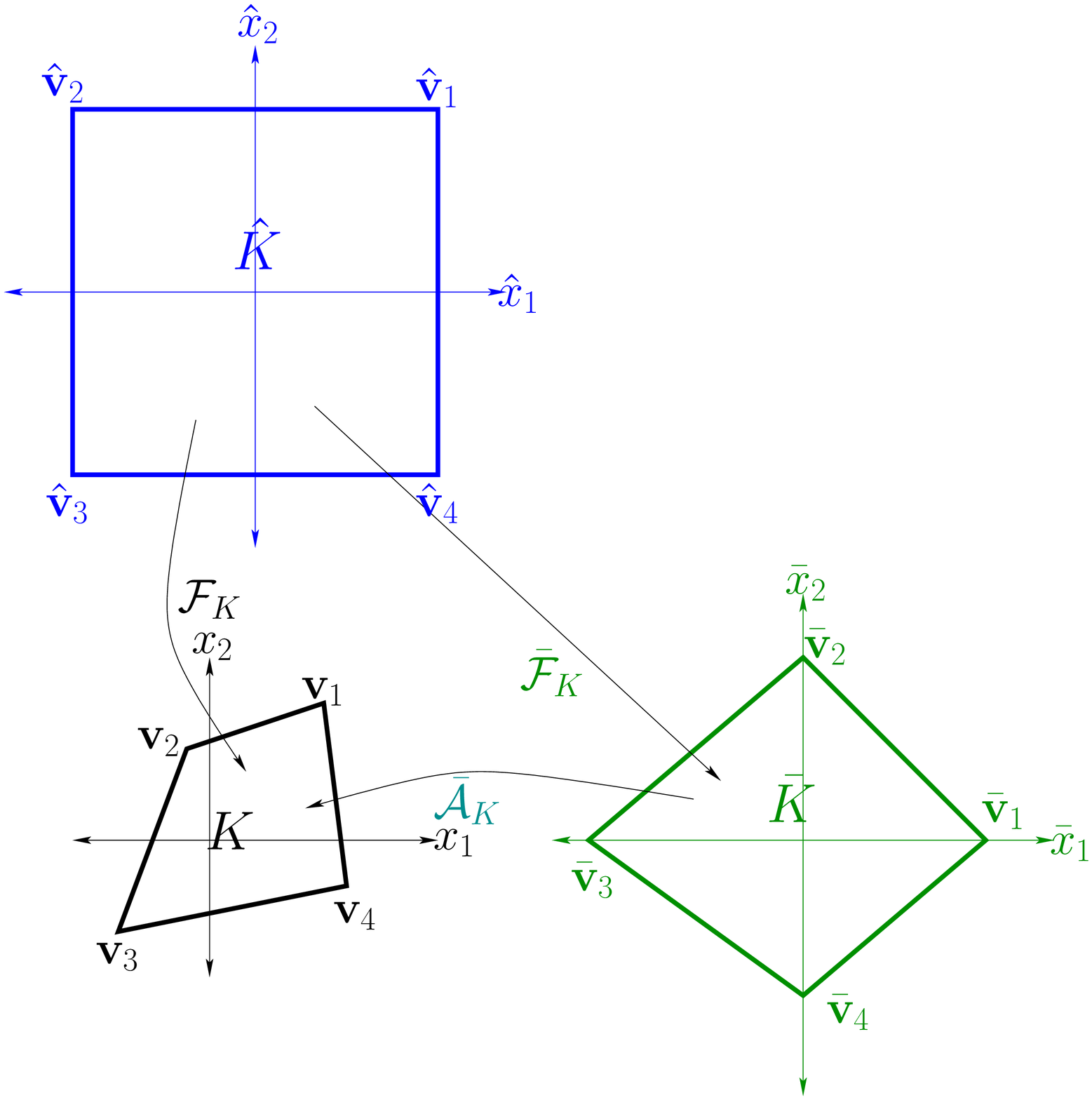}
\caption{A bilinear map $\mcF_K$ from $\hat{K}$ onto $K$,  a bilinear map
  $\bar{\mathcal F}_K$ from $\hat{K}$ onto $\bar{K}$, and an affine map $\bcA_K$ from $\bar{K}$ onto $K$.}
\label{fig:Bil}
\end{figure}
Define $\ell_j,j=1,2$ by
\begin{eqnarray}\label{eq:bell}
\ell_1(\mbx) = \frac{ (\mbx-\bv_3)\times(\bv_1-\bv_3)}{ (\bv_2-\bv_3)\times(\bv_1-\bv_3)},\q
\ell_2(\mbx) = \frac{ (\mbx-\bv_4)\times(\bv_2-\bv_4)}{ (\bv_1-\bv_4)\times(\bv_2-\bv_4)},
\end{eqnarray}
where $\mathbf a\times\mathbf b$ denotes the cross product of vectors
$\mathbf a$ and $\mathbf b.$
Then
$\ell_1(\mbx)=0$ and $\ell_2(\mbx)=0$ are the equations of lines satisfying
\begin{equation}
  \ell_1(\bv_1)=\ell_1(\bv_3)=0,
  \ell_1(\bv_2)=1;\,
  \ell_2(\bv_2)=\ell_2(\bv_4)=0,
\ell_2(\bv_1)=1.
\end{equation}
Then the intermediate quadrilateral $\bar K$ of MCL type is
defined with the following four vertices:
\begin{equation}\label{convex_poly}
\bar{\bv}_1={1\choose 0}, \bar{\bv}_2={0\choose 1},
\bar{\bv}_3={\bh_1\choose 0},
\bar{\bv}_4={0\choose\bh_2}, \text{ with } \bh_1 = \ell_2(\bv_3)<0,
\bh_2:=\ell_1(\bv_4)<0.
\end{equation}
We will use $(x_1, x_2)$ and $(\bar x_1,\bar x_2)$ for the
notations for coordinates in $K$ domain and $\bar K$ domain, respectively.
The following proposition is useful for our analysis.
\begin{proposition}
  Let  $\bcA_K$ be an affine map defined by
\begin{eqnarray}\label{eq:cAK-2}
  \bcA_K\bbx = \begin{bmatrix} \bcA_K^{(1)} &
\bcA_K^{(2)} \end{bmatrix}\bbx + \bxi,
\end{eqnarray}
where the three column vectors are defined by
\begin{eqnarray}\label{eq:cAK-3}
\bcA_K^{(1)} = \frac{\bv_1-\bv_3}{1-\bh_1},\quad
\bcA_K^{(2)} = \frac{\bv_2-\bv_4}{1-\bh_2},\quad
\bxi= \frac{\bv_3-\bh_1\bv_1}{1-\bh_1} =
  \frac{\bv_4-\bh_2\bv_2}{1-\bh_2}.
\end{eqnarray}
Then $\bcA_K: \mbR^2 \to \mbR^2$ is an affine map which maps $\bar K$
onto $K.$
Moreover, the affine map
$\bcA_K^{-1}: K\to \bar K$ by
\begin{eqnarray}\label{eq:cAK-31}
{\bar x_1\choose\bar x_2} = \bcA_K^{-1}\mbx =
{ \ell_2(\mbx) \choose\ell_1(\mbx)}.
\end{eqnarray}
maps $K$ onto $\bar K$ with four corresponding vertices.
\end{proposition}

\begin{proof}
  Let $\bxi$ be the intersection of the line passing through
  $\bv_1$ and $\bv_3$ and the line passing through
  $\bv_2$ and $\bv_4.$ Then,
  we have
  \begin{eqnarray}
    \bh_1(\bv_1-\bxi) = (\bv_3 -\bxi),\q
    \bh_2(\bv_2-\bxi) = (\bv_4 -\bxi),
  \end{eqnarray}
  from which it follows that
    \begin{eqnarray}\label{eq:bxi}
      \bxi = \frac{\bv_3-\bh_1\bv_1}{1-\bh_1} = \frac{\bv_4-\bh_2\bv_2}{1-\bh_2}.
    \end{eqnarray}
    Set $\bcA_K^{(j)}= \bv_j - \bxi,j =1,2.$ Then, owing to
    \eqref{eq:bxi}, the column vectors and the shift
    vector
    $\bxi$ in \eqref{eq:cAK-2} are given as in \eqref{eq:cAK-3}.
    It is immediate to verify that $\bcA_K\bar\bv_j = \bv_j,j=1,\cdots,4.$    
  \end{proof}

Then the nonparametric MCL nonconforming quadrilateral element 
is defined as follows.
\begin{definition}[Meng-Cui-Luo nonconforming quadrilateral element]
   \label{def:meng-cui-luo}
Define the intermediate reference element
$(\bar{K},\bar{P}_{\bar{K}}, \bar{\Sigma}_{\bar{K}})$ by
(1) $\bar{K}$ is a typical quadrilateral of MCL type;$\q$
(2) $\bar{P}_{\bar{K}}^{npMCL}=P_1(\bar K)\oplus \Span\{\bar{x}_1\bar{x}_2 \};\q$
(3) $\bar{\Sigma}^{npMCL}_{\bar{K}}=\{\sig_e^{(i)},\,e\in\cE(\bar K)\}.$

Then the global nonparametric MCL nonconforming quadrilateral
element space can be defined as follows:
\begin{eqnarray*}
  &\NC^{npMCL}_{h}=\{ v \in L^2(\O) \suchthat v|_K = \bar v\circ\bar
    A_K^{-1}\text{ for some } \bar v\in \bar P^{npMCL}_{\bar K}
    \foreach K \in \Tau_h,\\
  &\hspace{8cm} \sig_e^{(i)}(\jump{v}{e})=0 \foreach e \in \cE^i_h \},\\
&\NC^{npMCL}_{h,0}=\{v \in \NC^{npMCL}_{h} \suchthat \sig_e^{i}(v) =0 \foreach
~ e \in \cE_h^b \}.
\end{eqnarray*}

\end{definition}
\begin{remark} In \cite{meng2018new} the polynomial space
  $P_K$ is defined as 
  $P_K = \Span\{1,x_1,x_2,\ell_1(\mbx)\ell_2(\mbx)\}.$ But, due to the
  definitions of $\ell_1, \ell_2,$ and $\bcA_K,$
  it is evident to see the two definitions of $P_K$ are identical.
\end{remark}
\begin{remark} The Rannacher-Turek element is defined by using the
  reference element, and the basis function space $P_K$ in Definition \ref{def:rannacher-turek} consists of
  $P_1(K)\oplus \Span\{\hat\phi_4\circ \mcF_K^{-1}\},$ with
  $\hat\phi_4(\hat\mbx) = \hat x_1^2-\hat x_2^2.$
Since the components of $\mcF_K^{-1}$ are not polynomials,
  the function
  $\hat\phi_4\circ \mcF_K^{-1}(\mbx)$ is not a polynomial. In
  contrast, the MCL finite element space
  $P_K$ in  Definition \ref{def:meng-cui-luo} consists of quadratic
  polynomials since $\bcA_K^{-1}(\mbx)$ is an affine map.
\end{remark}

\begin{remark}
  The convexity condition on $K$ is that
\begin{equation}
 \bh_1 <0, \:  \bh_2 <0.
\end{equation}
\end{remark}

\section{Nonparametric DSSY elements in MCL-type domains}\label{sec:npDSSY}
In this section, we use the intermediate domain $\bar K$ of
MCL type to design a new nonconforming quadrilateral element
with MVP.

Similarly to \eqref{eq:til-mu} as in \cite{jeon2013class},
we aim to find a quartic polynomial $\bar{\mu}$
in $\bar{K}$ as follows:
\begin{equation}\label{eq:bar-mu}
\bar{\mu}(\bbx)=\bar\ell_1(\bbx) \bar\ell_2(\bbx)\bar{\sq}(\bbx),
\end{equation}
where
${\bar\ell}_1(\bbx) = \bar x_2, \bar{\ell}_2(\bbx) = \bar x_1,$ and
$\bar{\sq}(\bbx)$ is a suitable quadratic polynomial. (Recall that the
 coordinate indices for $\bar\ell_j$ and $\bar x_k$ are switched as in \eqref{eq:cAK-31}.)
We seek a quartic
polynomial $\bar{\mu}(\bbx)$ fulfilling the
MVP \eqref{eq:mvp} in $\bar{K}$.
 Denote by
$\bar{\mathbf{d}}_j:=\frac{\bar{\mathbf{v}}_{j+1}-\bar{\mathbf{v}}_{j}}{2}$
and $\bar{\mathbf{m}}_j$
the edge vector of $\bar{\mathbf{e}}_j$ and its midpoint
for $j=1,\cdots,4$ (with indices modulo 4).
An application of
the three-point Gauss quadrature formula:
\begin{eqnarray*}
\int_{-1}^1 f(t)~\operatorname{d}t \approx \frac89 f(0) + \frac59 (f(\xi) +
f(-\xi)),\quad  \xi =\sqrt{\frac35},
\end{eqnarray*}
which is exact for quartic polynomials,
simplifies MVP \eqref{eq:mvp} into the form
\begin{eqnarray}\label{eq:quadrature-2}
 \bar{\mu}(\bbg_{2j-1}) + \bar{\mu}(\bbg_{2j}) -2\bar{\mu}(\bbm_j) = 0, \quad j = 1, \cdots, 4,
\end{eqnarray}
where
$\bbg_{2j-1} = \bbm_j - \xi\bar{\mathbf{d}}_j,\quad
\bbg_{2j} = \bbm_j + \xi\bar{\mathbf{d}}_j, J=1,\cdots,4,
$
together with $\bbm_j,j = 1,\cdots,4,$
are the twelve Gauss points on the four edges. 
Notice that the equations of lines for edges $\bbe_j,j = 1,\cdots,4,$ are given
in vector notation as follows:
\[
\bbe_j(t) = \bbm_j + t\bar{\mathbf{d}}_j, \q t\in [-1,1].
\] 
Consider the quartic polynomial \eqref{eq:bar-mu} restricted to an edge
$\bbe_j(t), t\in[-1,1].$ Since $\bar{x}_1\bar{x}_2$ vanishes at the other two end points of each edge,
one sees that
\begin{eqnarray}\label{eq:inter-quad}
\bar{\ell}_1(\bbg_{2j-1})\bar{\ell}_2(\bbg_{2j-1}) =
\bar{\ell}_1(\bbg_{2j})\bar{\ell}_2(\bbg_{2j}) = (1-\xi^2)
\bar{\ell}_1(\bbm_{j})\bar{\ell}_2(\bbm_{j}), \quad \left(\xi = \sqrt{\frac35}\right).
\end{eqnarray}

A combination of \eqref{eq:quadrature-2} and \eqref{eq:inter-quad} yields that
\eqref{eq:mvp} holds if and only if the quadratic polynomial $\bar{\sq}$
satisfies
\begin{eqnarray}\label{eq:inter-quad-2}
 \bar{\sq}(\bbg_{2j-1}) + \bar{\sq}(\bbg_{2j}) -5\bar{\sq}(\bbm_j) = 0, \quad j = 1, \cdots, 4.
\end{eqnarray}

A standard use of symbolic calculation gives the general solution of
\eqref{eq:inter-quad-2} in the following one-parameter family in $\bar{c}$:
\begin{eqnarray}\label{eq:sq}
  \bar{\sq}(\bbx) = \bar q(\barx_1;\bh_1) + \bar{c}\bar q(\barx_2;\bh_2),\,
    \bar{c} \in \mathbb{R},\q
  \text{where }\bar q(\barx;\bh)= \barx^2 - \frac{3}{10}(1+\bh)\barx +
  \frac{3}{20}\bh.
\end{eqnarray}
Define for each $\bar{c} \in \mathbb{R},$
\[
\bar{\mu}(\bbx;\bar{c}) =
\bar{\ell}_1(\bbx)\bar{\ell}_2(\bbx)\bar{\sq}(\bbx)
= \bar{x}_1\bar{x}_2\bar{\sq}(\bbx).
\]
where $\bar{\sq}$ by \eqref{eq:sq} depending on $\bar{c}$ as well as $\bbx$.

We are now in a position to define a class of {\it nonparametric nonconforming
elements on the intermediate quadrilaterals $\bar{K}$} with four degrees of freedom as follows.
(1) $\bar{K} = \cS_K(\bar K);\q$
(2) $\bar{P}^{npDSSY}_{\bar{K}} = P_1(\bar K)\oplus \Span\{\bar{\mu}(\barx_1,\barx_2;\bar{c}) \}$\q;
(3) $\bar{\Sigma}^{npDSSY}_{\bar{K}}=\{\sig_e^{(i)},\,e\in\cE(\bar K)\}=
\{\sig_e^{(m)},\,e\in\cE(\bar K)\}.$

The global nonparametric DSSY quadrilateral nonconforming element
spaces are defined similarly.

By the above construction it is apparent that
MVP holds. Owing to this property, it is simple to show the unisolvency of the element.
\begin{theorem}\label{thm:unisol}
Assume that $\bar{c}$ is chosen such that $\bh_1^2+\bh_1+1+\bar{c}(\bh_2^2+\bh_2+1)\neq 0.$ Then 
$(\bar{K},\bar{P}_{\bar{K}}, \bar{\Sigma}_{\bar{K}} )$ is unisolvent.
\end{theorem}

\begin{proof}
  Set
$\bar{\psi}_1(\bbx)=1,\bar{\psi}_2(\bbx)=\barx_1,\bar{\psi}_3(\bbx)=\barx_2,$
and $\bar{\psi}_4=\bar{\mu}(\barx_1,\barx_2;\bar{c}),$ and
define $A=(a_{jk}) \in M_{4\times4}(\mathbb{R})$ by
$a_{jk}= \bar\psi_k(\bbm_j), j,k=1,\cdots,4.$ Then,
  Using $\bar q(\frac12;\bh)=\frac1{10}$ and
  $\bar q(\frac{\bh}{2};\bh)=\frac{\bh^2}{10},$ we get
%
\begin{equation}\label{eq:matA}
A = \begin{bmatrix}
1 & \frac{1}{2} & \frac{1}{2}    & \frac{1}{40}(1+ \bar c)\\
1 & \frac{\bh_1}{2} & \frac12   & \frac{1}{40}\bh_1(\bh_1^2+ \bar c)\\
1 & \frac{\bh_1}{2} & \frac{\bh_2}{2}   & \frac{1}{40}\bh_1\bh_2(\bh_1^2+ \bar c\bh_2^2)\\
1 & \frac{1}{2} & \frac{\bh_2}{2}   & \frac{1}{40}\bh_2(1 + \bar c\bh_2^2)
 \end{bmatrix}
\end{equation}
with $\det(A)=-\frac{1}{160}(1-\bh_1)^2(1-\bh_2)^2 \big(\bh_1^2+\bh_1+1+\bar{c}(\bh_2^2+\bh_2+1)\big).$  Thus $A$ is nonsingular if and only if $\bh_1^2+\bh_1+1+\bar{c}(\bh_2^2+\bh_2+1)\neq 0.$
\end{proof}
From now on, we choose $\bar{c} =1$ to have
symmetry in \eqref{eq:sq}.
%
%
The basis functions can be easily constructed by using MVP. For
this,
as in the proof of Theorem \ref{thm:unisol},
set $\bar\psi_1(\barx)=1, \bar\psi_2(\barx)=\bar x_1, \bar\psi_3(\bar x)=\bar x_2,
\bar\psi_4(\bar x)=\bar\mu(\bar x,\bh)$ and seek the basis functions
\begin{eqnarray}\label{eq:basis-barphi}
\bar\phi_j(\bar x)= \sum_{k=1}^4 c_{jk}
  \bar\psi_k(\bar x),\,j=1,\cdots,4,
\text{ such that } \bar\phi_j(\bm_k)=\delta_{jk}.
\end{eqnarray}

\def\den{\operatorname{D}}
\def\num{\operatorname{N}}
\def\Nh{\operatorname{Nh}}
By computing the inverse of the matrix $A^{T}$ in \eqref{eq:matA},
we can represent the basis functions \eqref{eq:basis-barphi}
with the coefficients given by
\begin{eqnarray}
  (c_{jk}) = \frac2{(1-\bh_1)(1-\bh_2) }
\begin{bmatrix}
\frac{ \bh_1 \bh_2}{2} &
 - \num(1,2) &
 - \num(2,1) &
 \frac{20}{\den} \\
 -\frac{\bh_2}{2} &
  \num(1,2) &
  \Nh(2) &
 -\frac{20}{\den}\\
\frac12 &
 - \Nh(1)&
 - \Nh(2)&
 \frac{20}{\den}\\
 -\frac{\bh_1}{2} &
  \Nh(1)&
  \num(2,1)&
  -\frac{40}{\den}
\end{bmatrix},
\end{eqnarray}
where
$
\den = 2+\bh_1+\bh_2+\bh_1^2+\bh_2^2,\,
\num(j,k) = \frac{ (\bh_j^2 + \bh_j + \bh_k^2 + 1) \bh_k }{\den},\,
\Nh(j)=\frac{\bh_j^2 + \bh_j + 2}{\den},\, j,k =1, 2.
$
\section{Effects of numerical integration}\label{sec:effects}
Consider the following elliptic boundary problem 
\begin{subeqnarray}\label{eq:model-elliptic}
-\div \left( \bkap(x) \grad u \right) &=& f  \text{ in } \Omega,\\
u &=& 0  \text{ on }  \partial \Omega,
\end{subeqnarray}
where $\Omega \subset \mathbb{R}^n$ is a polygonal domain and $\bkap=\big(\bkap_{ij}(x)\big)$ is a
symmetric matrix with smooth functions $\bkap_{ij}\in W^{1,\infty,h}(\O)$
such that there are constants $\alpha_K>0$ such that
\begin{equation*}
\sum_{i,j=1}^n \kap_{ij}(\mbx) \xi_i \xi_j \geq \alpha_K
|\xi|^2,\,\mbx\in K\forany K\in\Tau_h.
\end{equation*}

The variational form of \eqref{eq:model-elliptic} is given by finding 
$u \in H_0^1(\Omega)$ such that
\begin{equation}\label{eq:model-weakform}
a(u,v) = F(v), \q v\in H_0^1(\Omega),
\end{equation}
where $a(u,v)= \int_\Omega \bkap\grad u \cdot \grad v dx$ and 
$F(v)=\int_\Omega fv dx.$
Consider the nonparametric DSSY space $\NC^{npDSSY}_h$ defined by
$(K,P_K, \Sigma_K)$ in the previous section.
Then the finite element approximation $u_h\in \NC^{npDSSY}_h$ of
\eqref{eq:model-weakform} is defined as the solution of discrete problem
\begin{equation}\label{eq:model-discreteform}
a_h(u_h,v_h) = F_h(v_h), \q v \in \NC^{npDSSY}_h
\end{equation}
where $a_h(u,v)= \sum_{K \in \mathcal{T}_h} \int_K \bkap\grad u \cdot
\grad v dx$ and $F_h(v)=\sum_{K \in \mathcal{T}_h} \int_K fv dx.$

The energy error estimate for nonconforming methods are as
provided in \cite{dssy-nc-ell}.
\begin{theorem}\cite{dssy-nc-ell}
Assume that $u$ and $u_h$ are the solutions of \eqref{eq:model-weakform} and \eqref{eq:model-discreteform}, respectively. Then we have following error estimate
\begin{equation}
|u-u_h|_{1,h} \leq C h |u|_{2,\Omega}.
\end{equation}
\end{theorem}

Our aim in this section is to find sufficient conditions on numerical quadrature rules
in approximating the stiffness matrix and load vector based on $\NC^{npDSSY}_h$.
Notice that our nonconforming elements on $K$ are constructed
via the affine map $\bar{\cA}_K$ form the reference element $\bar{K}$ onto
$K$.
Thus it is natural to construct a quadrature formulae on $\bar{K}$,
which is defined with positive weights $\bar{\omega}_{\bar K,l}$
and nodes $\bar{\mathbf{b}}_{\bar K,l}$, $l=1,\cdots,L,$ by
\begin{equation}\label{eq:quad_ref}
\int_{\bar{K}} \bar{\phi}(\bar{\bx})d\bar{\bx} \approx \sum_{l=1}^L
\bar{\omega}_{\bar K, l} \bar{\phi}({\bar{\mathbf b}}_{\bar K, l}).
\end{equation}
Denote  by $J_{\bar{\cA}_K}$ the Jacobian matrix of $\bar{\cA}_K$,
and observe that
\begin{equation*}
\int_K \phi(\bx) d\bx = |\det(J_{\bar{\cA}_K})| \int_{\bar{K}} \bar{\phi}(\bar{\bx})d\bar{\bx}.
\end{equation*}
It induces the quadrature formulae on $K$ from \eqref{eq:quad_ref}, which is given by
\begin{equation}\label{eq:quad_physical}
\int_K \phi(\bx) d\bx \approx \sum_{l=1}^L \omega_{K,l} \phi(\mathbf{b}_{K,l}),
\end{equation}
where $\omega_{K,l} = |\det(J_{\bar{\cA}_K})|\, \bar{\omega}_{\bar K,l}$ and
$\mathbf{b}_{K,l} = \bar{\cA}_K\bbb_{\bar K,l}$.
Suppose that the discrete problem \eqref{eq:model-discreteform}
is approximated by above  quadrature formulae. Then $\bar{u}_h$, the Galerkin approximation
with quadrature, is defined as the solution of approximate problem: to
find
$\bar u_h\in \NC^{npDSSY}_h$ such that
\begin{equation}\label{eq:model-approxi}
\bar{a}_h(\bar{u}_h,\bar{v}_h) = \bar{F}_h({v}_h), \q {v}_h \in \NC^{npDSSY}_h,
\end{equation} 
where 
\begin{align}
\bar{a}_h(u,v) = \sum_{K \in \mathcal{T}_h} \omega_{K,l} (\bkap\grad u
  \cdot \grad v )(\mathbf{b}_{K,l});\q 
\bar{F}_h(v) = \sum_{K \in \mathcal{T}_h} \omega_{K,l} (fv)(\mathbf{b}_{K,l}). 
\end{align}
Define the quadrature error functionals to estimate the effect of numerical integration by
\begin{subeqnarray}\label{eq:quad-rule}
\bar{E}_{\bar K}(\bar{\phi}) &=& \int_{\bar{K}}
\bar{\phi}(\bar{\bx})d\bar{\bx} - \sum_{l=1}^L \bar{\omega}_{\bar K,l}
\phi(\bbb_{\bar K, l}), \slabel{eq:quad-rulea} \\
E_K(\phi) &=& \int_K \phi(\bx) d\bx - \sum_{l=1}^L \omega_{K,l}
\phi(\mathbf{b}_{K,l}),
\slabel{eq:quad-ruleb}
\end{subeqnarray}
where above two error functionals are related as follows:
\begin{equation*}
E_K(\phi) = |\det(J_{\bar{\cA}_K})| \bar{E}_{\bar K}(\bar{\phi}).
\end{equation*}
The following Bramble-Hilbert Lemma is essential for our argument. 
\begin{lemma}\label{lem:Bramble-Hilbert}\cite[Theorem 28.1, p.198]{ciarlet1991finite}
Let $D \subset \mathbb{R}^n$ be a domain with a Lipschitz continuous
boundary. Suppose that $L$ is a continuous linear mapping on
$W^{k+1,q}(D)$ for some integer $k \geq 0$ and $ 1\le q \le \infty$. If
 \begin{equation}
L(\phi)=0 \q \forall \phi \in P_k(D),
\end{equation} 
then there exists a constant $C(D)$ such that
\begin{equation}
|L(v)| \leq C(D) \|L\|_{k+1,q,D}' |v|_{k+1,q,D}\forany v\in W^{k+1,q}(D),
\end{equation}
where $ \|\cdot\|_{k+1,q,D}' $ denotes the norm in the dual space of $W^{k+1,q}(D).$
\end{lemma}

The following theorem estimates the effect of quadrature formulae on
the approximate bilinear form $\bar{a}_h,$ whose proof is essentially
identical to that of \cite[Theorem 28.2, p.199]{ciarlet1991finite}.
\begin{theorem}\label{thm:3.4}
Assume that $\bar{E}_{\bar K}(\bar{\phi}) = 0$ for any $\bar{\phi} \in \grad \bar{P}_{\bar{K}}$, where
\begin{equation*}
\grad\bar{P}_{\bar{K}}:=\Span\{\grad \bar{v} \; \big| \; \bar{v} \in \bar{P}_{\bar{K}} \}.
\end{equation*}
If $\bkap \in W^{1,\infty}(K),$
then there exists a constant $C>0$ such that
\begin{align}
\big| E_K(\bkap \grad u \cdot \grad v) \big|  \leq C h_K
  ||\bkap||_{1,\infty,K} |u|_{1,K} |v|_{1,K}
  \q \forall u,v \in P_K, 
\end{align}
where $h_K$ denotes the diameter of $K$.
\end{theorem}
We now establish the following conditional ellipticity estimate.
\begin{theorem}\label{thm:conditional-ell}
Assume that $\bar{E}_{\bar K}(\bar{\phi}) = 0$ for any $\bar{\phi} \in \grad \bar{P}_{\bar{K}}$, where
\begin{equation*}
\grad\bar{P}_{\bar{K}}:=\Span\{\grad \bar{v} \; \big| \; \bar{v} \in \bar{P}_{\bar{K}} \}.
\end{equation*}
Assume that $\bkap \in W^{1,\infty}(\Tau_h).$ Then for
sufficiently small $h>0$,
the following ellipticity holds:
\begin{equation}\label{eq:app_ellipticity}
  \bar{a}_h(v,v) \geq  (\alpha- C h |\kap|_{1,\infty,h}) \cor{|v|_{1,h}^2}
  \forany v \in \NC^{npDSSY}_h.
\end{equation}
Moreover, the conditional ellipticity coefficient in \eqref{eq:app_ellipticity}
can be given as
$\min_{K\in\Tau_h}(\alpha_K-C_Kh_K|\kap|_{1,\infty,h}).$
\end{theorem}
\begin{proof}
Let $v\in\NC^{npDSSY}_h$ be arbitrary. Then
by the triangle inequality, the ellipticity, and Theorem
\ref{thm:conditional-ell}, we have
  \begin{subeqnarray*}
    \bar a_h(v,v) &\ge& \sum_{K\in\Tau_h} \left[ a_K(v,v) -|E_K(\kap\grad u, \grad u)|\right] \\
&\ge & \min_{K\in\Tau_h}(\alpha_K-C_Kh_K|\kap|_{1,\infty,h})\cor{|v|_{1,h}^2}\\
&\ge & (\alpha- C h |\kap|_{1,\infty,h})\cor{|v|_{1,h}^2}.
  \end{subeqnarray*}
This completes the proof.  
\end{proof}
Next, we estimate the effect of numerical integration on the right hand
side linear functional $\bar{F}_h,$
which is essentially identical to the case of $k=1$ of \cite[Theorem 28.3,
p.201]{ciarlet1991finite}.
\begin{theorem}\label{thm:3.5}
Suppose that $\bar{E}_{\bar K}(\bar{\phi}) = 0$ for any $\phi \in
P_0(\bar{K})$.
Then for arbitrary $f \in W^{1,\infty}(\Omega)$ and $\phi \in P_K$,
there exists a constant $C>0$ such that
\begin{equation*}
\big| E_K(f\phi) \big| \leq C h_K
|K|^{1/2}||f||_{1,\infty,K} ||\phi||_{1,K}\forany
\phi\in P_K,
\end{equation*}
where $h_K$ denotes the diameter of $K$.
\end{theorem}

Finally, the effect of numerical integration is obtained by combining the above results.
\begin{theorem}
  Let $u$ and $\bar{u}_h$ be the solutions of \eqref{eq:model-weakform}
  and \eqref{eq:model-approxi}, respectively. Assume that
  $\bar{E}_{\bar K}(\bar{\phi}) = 0$ for any
  $\bar{\phi} \in \grad \bar{P}_{\bar{K}}.$ Assume that $\bkap\in W^{1,\infty}(\Tau_h).$
 Then, for sufficiently small $h>0,$ we have the following error estimate
\begin{equation*}
|u-\bar{u}_h|_{1,h} \leq C \frac{h}{1-h\|\bkap\|_{1,\infty,h}} \big(||\bkap||_{1,\infty,h} |u|_{2,\Omega} + ||f||_{1,\infty,h} \big).
\end{equation*}
\end{theorem}

\begin{proof}
  We exploit the conditional uniform ellipticity of $\bar{a}_h$.
  Let $u_h$ be the solutions of \eqref{eq:model-discreteform}.
  Then we have, for sufficiently small $h>0,$ for any $v_h \in\NC^{npDSSY}_h$,
  \begin{align*}
(\alpha-C h\|\bkap\|_{1,\infty,h}) |\bar{u}_h-v_h|_{1,h}^2 &\leq \bar{a}_h(\bar{u}_h-v_h,\bar{u}_h-v_h) \\
&=\bar{a}_h(u_h-v_h,\bar{u}_h-v_h) + \bar{a}_h(\bar{u}_h-u_h,\bar{u}_h-v_h) \\
&=\bar{a}_h(u_h-v_h,\bar{u}_h-v_h) + \big(\bar{F}_h(\bar{u}_h-v_h)-\bar{a}_h(u_h,\bar{u}_h-v_h) \big) \\
&=\bar{a}_h(u_h-v_h,\bar{u}_h-v_h) + \big(a_h(u_h,\bar{u}_h-v_h)-\bar{a}_h(u_h,\bar{u}_h-v_h) \big) \\
& \qq\qq\qq\qq\qq\, + \big(\bar{F}_h(\bar{u}_h-v_h) - F_h(\bar{u}_h-v_h) \big).
\end{align*}
Denote by $w_h:=\bar{u}_h-v_h$. It follows that
\begin{align*}
(\alpha-C h\|\bkap\|_{1,\infty,h}) |\bar{u}_h-v_h|_{1,h} &\leq C|u_h-v_h|_{1,h} + \frac{|a_h(u_h,\bar{u}_h-v_h)-\bar{a}_h(u_h,\bar{u}_h-v_h)|}{|\bar{u}_h-v_h|_{1,h}} \\
&\q\qq\qq\qq + \frac{|F_h(\bar{u}_h-v_h) - \bar{F}_h(\bar{u}_h-v_h)|}{|\bar{u}_h-v_h|_{1,h}} \\
&\leq C \inf_{v_h \in \NC^{npDSSY}_h} |u_h-v_h|_{1,h} \\
&+ \sup_{w_h \in \NC^{npDSSY}_h} \Big(\frac{|a_h(u_h,w_h)-\bar{a}_h(u_h,w_h)|}{|w_h|_{1,h}} + \frac{|\bar{F}_h(w_h) - F_h(w_h)|}{|w_h|_{1,h}} \Big).
\end{align*}
If we take $v_h=u_h$, the above inequality is simplified to
\begin{align}\label{eq:3.22}
|\bar{u}_h-u_h|_{1,h} &\leq \frac{1}{\alpha-C h\|\bkap\|_{1,\infty,h}} \sup_{w_h \in \NC^{npDSSY}_h} \Big(\frac{|a_h(u_h,w_h)-\bar{a}_h(u_h,w_h)|}{|w_h|_{1,h}} + \frac{|\bar{F}_h(w_h) - F_h(w_h)|}{|w_h|_{1,h}} \Big).
\end{align}
It remains to estimate the above two consistency error terms. For the
first term,
using Theorem \ref{thm:3.4}, we have
\begin{align}\label{eq:3.23}
|a_h(u_h,w_h)-\bar{a}_h(u_h,w_h)| & \leq \sum_{K \in \mathcal{T}_h} \big| E_K(\bkap\grad v_h \cdot \grad w_h) \big|\nn \\
& \leq \sum_{K \in \mathcal{T}_h} h_K ||\bkap||_{1,\infty,K} |v_h|_{1,K} |w_h|_{1,K}\nn \\
& \leq C h  ||\bkap||_{1,\infty,h} |u_h|_{1,h} |w_h|_{1,h} \nn\\
& \leq C h  ||\bkap||_{1,\infty,h} |u|_{2,\Omega} |w_h|_{1,h},
\end{align}
where the last inequality is obtained by the following estimate:
\begin{align*}
|u_h|_{1,h} &\leq |u|_{1,\Omega} + |u-u_h|_{1,h}\\
&\leq |u|_{1,\Omega} + C h |u|_{2,\Omega} \leq C |u|_{2,\Omega}.
\end{align*}
For the second consistency error term in \eqref{eq:3.22},
Theorem \ref{thm:3.5} applies:
\begin{align}\label{eq:3.24}
|\bar{F}_h(w_h)-F_h(w_h)| & \leq \sum_{K \in \mathcal{T}_h} \big| E_K(f w_h)\big| \nn\\
& \leq C \sum_{K \in \mathcal{T}_h} h_K |K|^{1/2}||f||_{1,\infty,K} |w_h|_{1,K} \nn\\
& \leq C h |\Omega|^{1/2} ||f||_{1,\infty,h} |w_h|_{1,h}.
\end{align}
The theorem follows by combining \eqref{eq:3.22}--\eqref{eq:3.24}
with the triangle inequality. That is, for sufficiently small $h>0,$
\begin{align*}
|u-\bar{u}_h|_{1,h} & \leq |u-u_h|_{1,h} + |u_h-\bar{u}_h|_{1,h} \\
&\leq C h |u|_{2,\Omega} + C \frac{h}{1- h\|\bkap\|_{1,\infty,h} } \big(||\bkap||_{1,\infty,h} |u|_{2,\Omega} + |\Omega|^{1/2}||f||_{1,\infty,h} \big) \\
&\leq C \frac{h}{1- h\|\bkap\|_{1,\infty,h} } \big(||\bkap||_{1,\infty,h} |u|_{2,\Omega} + ||f||_{1,\infty,h} \big).
\end{align*}
\end{proof}

\section{Construction of quadrature formulae}\label{sec:const-formula}
In this section we construct quadrature formulae for the nonparametric
DSSY element defined in Section \ref{sec:npDSSY}.
\subsection{Quadrature formula on $\bar{K}$}
In \cite{meng2018new}, the basis functions are at most of degree two
so that the quadrature formulae of degree two is found.
However our element has high-order degree polynomial basis to fulfill MVP,
and thus we require some other quadrature formulae.
Following the analysis of previous section, we may find quadrature
formulae exact for functions in $\bar{\cQ}$, which is defined as
\begin{equation*}
\bar{\cQ} := \Span\{\grad \bar{u} \; \big|\;
\bar{u}, \bar{v} \in \bar{P}_{\bar{K}} \}=\Span\{1,
\frac{\p\bar\mu}{\p\bar x_1},
\frac{\p\bar\mu}{\p\bar x_2}\}
\end{equation*}
to preserve the order of convergence.
From now on, we choose $\bar c=1,$ the generalization being trivial
to include the other cases.

We quote the formula from \cite[(1) p.330]{meng2018new}:
\begin{align}\label{eq:exact-int}
\int_{\bar{K}} \barx_1^j \barx_2^k d\bar{\bx} = \sum_{k=1}^4
 \int_{\bar{T}_k} \barx_1^j \barx_2^k d\bar{\bx} 
=\frac{j!k!}{(2+j+k)!}(1-{\bar h_1}^{j+1})(1-{\bar h_2}^{k+1}).
\end{align}
We have the area
$$|\bar K|=\frac{(1-\bar h_1)(1-\bar h_2)}{2}.$$
For the sake of notational simplicity, we write
\begin{eqnarray*}
  \dbh_1:=1+\bar{h}_1,\quad \dbh_2:=1+\bar{h}_2,\quad \dbh_1 < 1,\, \dbh_2 < 1.
\end{eqnarray*}
Clearly, $\dbbh_c=\frac13{\dbh_1\choose \dbh_2}$ is the barycenter of $\bar K.$

\subsection{One-point quadrature rule of precision 1}
The obvious one-point quadrature weight and point are given by
\begin{eqnarray}\label{eq:1pt-form}
\bar{\omega}_{1} = |\bar K|,\q
\bar\bxi^{(1)} = \dbbh_c.
\end{eqnarray}
\subsection{Two-point and three-point quadrature rules of precision 1}
We seek two-point and three-point quadrature rules \eqref{eq:quad_ref} with
equal weights at one stroke. The weights are then given by
$\omega_{\bar K,l}=\frac{|K|}{L}, l=1,\cdots, L,$ for $L=2,3.$
The Gauss points $\bbb_{\bar K,l},l=1,\cdots, L,$ are assumed to be
geometrically symmetric with respect to the barycenter
$\dbbh_c.$ Hence we seek $\bxi^{(L)} = {\xi_1^{(L)}\choose
  \xi_2^{(L)}}$ such that
\begin{equation}\label{eq:symm-quad}
  \int_{\bar{K}} \bar{\phi}(\bar{\bx})d\bar{\bx} \approx
\begin{cases} \frac{|K|}{2} \left( \bar{\phi}(\dbbh_c + \bxi^{(2)})
 + \bar{\phi}(\dbbh_c - \bxi^{(2)}) \right) & \text{for }L=2,\\
  \frac{|K|}{3} \left(
\bar{\phi}(\dbbh_c) + \bar{\phi}(\dbbh_c + \bxi^{(3)})  + \bar{\phi}(\dbbh_c
  - \bxi^{(3)})\right)
& \text{for }L=3,
\end{cases}
\end{equation}
which are exact for
\begin{subeqnarray}\label{eq:mu_der}
\frac{\p\bar\mu}{\p\bar{x}_1} &=&
3\bar x_1^2\bar x_2 + \bar x_2^3 - \frac3{10}( 2\dbh_1\bar x_1\bar x_2 + \dbh_2\bar x_2^2) + \frac3{20}(\dbh_1+\dbh_2-2)\bar x_2,\\
\frac{\p\bar\mu}{\p\bar{x}_2} &=&
\bar x_1^3 + 3\bar x_1\bar x_2^2 - \frac3{10}( \dbh_1\bar x_1^2 +
2\dbh_2\bar x_1\bar x_2) + \frac3{20}(\dbh_1+\dbh_2-2)\bar x_1.
\end{subeqnarray}
Then, by utilizing \eqref{eq:exact-int}, the
exactness of \eqref{eq:symm-quad} for \eqref{eq:mu_der} implies that
$(\xi^{(L)}_1, \xi^{(L}_2),L=2,3,$ turn out to be the solutions
$(X,Y)$ of the quadratic equations:
\begin{subeqnarray}\label{eq:poly-sym}
10 \dbh_2 X^2  + 14 \dbh_1 XY + 7 \dbh_2 Y^2 &=&
Lr(\dbh_1,\dbh_2),\slabel{eq:poly-syma}  \\
7 \dbh_1 X^2  + 14 \dbh_2 XY + 10 \dbh_1 Y^2 &=& Lr(\dbh_2,\dbh_1),\slabel{eq:poly-symb}
\end{subeqnarray}
where
$r(u,v)=\frac{5}{4} v \left(\frac2{90} u^2- \frac4{10} u + \frac{185}{999} v^2 -
  \frac6{10} v + 1\right).$
Since \eqref{eq:poly-sym} is a symmetric system of homogeneous
quadratic equations, it is easy to see the following lemma.
\begin{lemma}\label{lem:sym-sols}
 The four solutions of \eqref{eq:poly-sym} are given by
  $\pm (g_1(\dbh_1,\dbh_2),g_1(\dbh_2,\dbh_1)),$
and \\ $\pm(g_2(\dbh_1,\dbh_2),g_2(\dbh_2,\dbh_1))$ for some functions
$g_1(u,v)$ and $g_2(u,v).$
\end{lemma}
Denote by $r_1$ and $r_2$ the right hand sides of \eqref{eq:poly-sym}.
Then a judicial use of symbolic software (for example, Julia or Matlab)
with some cook-ups
gives the following two pairs of solutions $(X,Y)$:
\begin{subeqnarray}\label{eq:solXY}
X^{(1)}(\dbh_1,\dbh_2) &=&-\frac{T_4+\sqrt{T_3}}{T_1}Y^{(1)}(\dbh_1,\dbh_2),\q
Y^{(1)}(\dbh_1,\dbh_2)=\sqrt{\frac{T_5+T_6 \sqrt{T_3} }{T_2}};\\
X^{(2)}(\dbh_1,\dbh_2) &=&-\frac{T_4-\sqrt{T_3}}{T_1}Y^{(2)}(\dbh_1,\dbh_2),\q
Y^{(2)}(\dbh_1,\dbh_2)=\sqrt{\frac{T_5-T_6 \sqrt{T_3}  }{T_2}},
\end{subeqnarray}
where
$
T_1 = 7r_1\dbh_1 - 10r_2\dbh_2,\,
T_2 = 13720\dbh_1^4 - 26603\dbh_1^2\dbh_2^2 + 13720\dbh_2^4,\,
T_3 = -70(r_1^2\dbh_1^2 + r_2^2\dbh_2^2)+ 49(r_1^2\dbh_2^2+r_2^2\dbh_1^2) + 51r_1r_2\dbh_1\dbh_2 ,\,
T_4 = 7(r_1\dbh_2-r_2\dbh_1),\,
T_5 = -1043r_1\dbh_1^2\dbh_2 + 980r_1\dbh_2^3 + 686r_2\dbh_1^3 - 470r_2\dbh_1\dbh_2^2,\,
T_6 = 14(7\dbh_1^2 - 10\dbh_2^2 ).
$
Owing to the symmetries of $T_2$ and $T_3$ with respect to $\dbh_1$
and $\dbh_2$, one can check by using symbolic software, again, that
$X^{(j)}(\dbh_1,\dbh_2)=Y^{(j)}(\dbh_2,\dbh_1),j=1,2,$ which confirms 
Lemma \ref{lem:sym-sols}.
Obviously, $T_2>0$ unless $\dbh_1=0$ and $\dbh_2=0.$
Among the two pairs of solutions $(X^{(j)},Y^{(j)}),j=1,2,$ in \eqref{eq:solXY}, we choose
one that is closer to the origin $(0,0)$ if $\dbbh_c+{X^{(j)}\choose
  Y^{(j)}}$ is in $\bar K$ to 
increase numerical stability.
%
%

In the case where $\dbh_1=0$ or $\dbh_2=0,$ $T_1$ vanishes, and
therefore, the above formula is unstable. To deal with this,
first, assume the case of $\dbh_2=0.$ 
Then, the polynomial equations corresponding
to \eqref{eq:poly-sym} are simplified as follows:
\begin{subeqnarray}\label{eq:poly-sym-k=0}
1008 \dbh_1 (2-\dbh_1) XY &=&0,
\slabel{eq:poly-sym-k=0a}  \\
252X^2 + 360Y^2 &=& 25 \dbh_1^2 - 81 \dbh_1 + 135.
\slabel{eq:poly-sym-k=0b}
\end{subeqnarray}
From \eqref{eq:poly-sym-k=0a} and \eqref{eq:poly-sym-k=0b}, we have
either (i) $X^{(1)}=0, Y^{(1)}=\sqrt{\frac{25 \dbh_1^2 - 81 \dbh_1^2 + 135}{360}}$ or
 (ii) $Y^{(2)}=0, X^{(2)}=\sqrt{\frac{25 \dbh_1^2 - 81 \dbh_1^2 + 135}{252}}.$
 Among these two pair
 $(X^{(j)},Y^{(j)}),j=1,2,$ the suitable choice is made as above.

The case of $\dbh_1=0$ is treated similarly by rotation of the case of $\dbh_2=0.$

\section{Numerical examples}
In this section some numerical results are reported to confirm the
theoretical parts about the quadrature developed in the previous
sections.
\begin{example}\label{ex:1-2}
For the numerical example, consider the elliptic boundary value
problem \eqref{eq:model-elliptic} on $\O=(0,1)^2$ and
$\kappa(\textbf{x})=1+(1+x_1)(1+x_2)+\epsilon\sin(10\pi x_1)\sin(5\pi
x_2).$ The source term $f$ is generated by the exact solution
\begin{equation*}
    u(x_1,x_2) = \sin(3\pi x_1)x_2(1-x_2) +
    \epsilon\sin\left(\frac{\pi
        x_1}{\epsilon}\right)\sin\left(\frac{\pi
        x_2}{\epsilon}\right),\quad \epsilon=0.2. 
\end{equation*}
\end{example}
The above problem is solved by using the nonparametric DSSY element
constructed in Section \ref{sec:npDSSY}. 
Let $\bar u_h$ denote the nonparametric DSSY
Galerkin approximation to $u$ by using any specific Gauss quadrature
rule \eqref{eq:model-approxi}.
The meshes used in the numerical example are
perturbed from $(N\times N)$ uniform rectangles as follows.
The random meshes  $x_{jk},j,k=1,\cdots,N-1,$ are obtained by
perturbing the uniform mesh points $(j,k)h,h=\frac1{N}$ with randomly
by $r_1$ and $r_2$ such that $x_{jk} = (j+r_{jk,1},k+r_{jk,2})h$ with
$|r_{jk,l}| \le r, l=1,2.$ Here, $r=0.2$ was chosen. The linear systems are solved by the Conjugate
Gradient method with tolerance $10^{-7}$ for residuals.
The errors and reduction ratios with random perturbation are averaged with
20 ensembles, but the number of ensembles can be arbitrarily
increased.
Our 2--point and 3--point Gauss quadrature rules are compared with
the standard $2\times 2$--point and $3\times 3$--point tensor product Gauss
quadrature rules in computing the stiffness matrix and the right hand
side vectors. In order to make the comparison fair to the above four
different quadrature rules, all errors are calculated by using
$3\times 3$--point tensor product Gauss quadrature.
\begin{table}[htbp]
\centering
\begin{tabular}{|c|c|c|c|c|c|c|c|c|}
\hline
\multirow{2}{*}{$N$} & \multicolumn{4}{c|}{$2\times 2$ Gauss quadrature}
  & \multicolumn{4}{c|}{$3\times 3$ Gauss quadrature}\\  \cline{2-9}
                     & $|\bar{u}_h-u|_{1,h}$ & order &
                                                       $\|\bar{u}_h-u\|_0$
                                             & order &
                                                       $|\bar{u}_h-u|_{1,h}$
                                             & order &
                                                            $\|\bar{u}_h-u\|_0$ & order \\ \hline
4   & 4.37  &           & 0.213      &      & 2.45     &        & 0.131   &      \\ \hline
8   & 4.17  & 0.688E-01 & 0.857E-01  & 1.31 & 1.72     & 0.508  & 0.426E-01 & 1.62\\ \hline
16  & 2.39  & 0.803     & 0.240E-01  & 1.84 & 0.926    &0.896   & 0.112E-01 & 1.92\\ \hline
32  & 1.29  & 0.888     & 0.661E-02  & 1.86 & 0.471    &0.975   & 0.286E-02 & 1.97\\ \hline
64  & 0.731 & 0.821     & 0.204E-02  & 1.69 & 0.237    &0.994   & 0.718E-03 & 1.99\\ \hline
128 & 0.476 & 0.619     & 0.825E-03  & 1.31 & 0.118    &0.998   & 0.180E-03 & 2.00\\ \hline
256 & 0.384 & 0.311     & 0.497E-03  & 0.732& 0.592E-01&0.999   & 0.450E-04 & 2.00\\ \hline
\end{tabular}
\caption{Numerical results of Example \ref{ex:1-2} on random meshes with
  tensor product Gauss quadratures: the broken $H^1(\O)$--seminorm and
  $L^2(\O)$--norm errors and their reduction orders are reported.}
\label{tab:ex1-2-1}
\end{table}

\begin{table}[htbp]
\centering
\begin{tabular}{|c|c|c|c|c|c|c|c|c|}
\hline
\multirow{2}{*}{$N$} & \multicolumn{4}{c|}{$2$--pt Gauss quadrature}
  & \multicolumn{4}{c|}{$3$--pt Gauss quadrature}\\ \cline{2-9}
                     & $|\bar{u}_h-u|_{1,h}$ & order &
                                                       $\|\bar{u}_h-u\|_0$
                                             & order &
                                                       $|u-\bar{u}_h|_{1,h}$ & order & $\|\bar{u}_h-u\|_0$ & order \\ \hline
      4  &    13.2 &         &  0.960    &         &     4.61     &         &   0.539     &          \\ \hline
      8  &    9.04 &   0.548 &  0.228    &   2.08  &      2.42    &  0.927  &    0.113    &   2.26   \\ \hline 
     16  &    4.15 &    1.12 &  0.457E-01&   2.32  &      1.15    &   1.08  &    0.208E-01&   2.44   \\ \hline 
     32  &    2.10 &   0.984 &  0.107E-01&   2.09  &     0.559    &   1.04  &    0.388E-02&   2.42   \\ \hline 
     64  &    1.06 &   0.979 &  0.263E-02&   2.03  &     0.279    &   1.00  &    0.835E-03&   2.21   \\ \hline 
    128  &   0.533 &   0.999 &  0.655E-03&   2.01  &     0.139    &   1.00  &    0.200E-03&   2.06   \\ \hline 
    256  &   0.267 &   0.999 &  0.163E-03&   2.00  &     0.698E-01&  0.997  &    0.494E-04&   2.02   \\ \hline 
\end{tabular}
\caption{Numerical results of Example \ref{ex:1-2} on random meshes with
  2-pt and 3-pt Gauss quadratures: the broken $H^1(\O)$--seminorm
  and $L^2(\O)$--norm errors and their reduction orders are reported.}
\label{tab:ex1-2-3}
\end{table}

As seen from the tables, the $2\times 2$ tensor product Gauss
quadrature is not sufficient to integrate the matrices in the
nonparametric DSSY element methods, while the $3\times 3$ tensor
product rule is sufficient. In the meanwhile, our new
2--pt Gauss quadrature rule is almost optimal, which gives better
numerical results than the $2\times 2$ product rule.
Observe that numerical errors obtained by our $3$--pt Gauss
quadrature rule are as good as those obtained by using the $3\times 3$
tensor product rule. In most calculations, the 2--pt Gauss quadrature
rule is acceptable.

\section*{Acknowledgments}
This research was supported in part by National Research Foundations
(NRF-2017R1A2B3012506 and NRF-2015M3C4A7065662).


\begin{thebibliography}{10}

\bibitem{achchab2014simple}
B.~Achchab, A.~Agouzal, and K.~Bouihat.
\newblock A simple nonconforming quadrilateral finite element.
\newblock {\em C. R. Acad. Sci. Paris, Ser. I.}, 352(6):529--533, 2014.

\bibitem{achchab2018new}
B.~Achchab, A.~Agouzal, and K.~Bouihat.
\newblock A new class of nonconforming finite elements for arbitrary order.
\newblock {\em Rendiconti Sem. Mat. Univ. Pol. Torino}, 76(2):11--17, 2018.

\bibitem{altmann-carstensen}
R.~Altmann and C.~Carstensen.
\newblock {$P_1$}-nonconforming finite elements on triangulations into
  triangles and quadrilaterals.
\newblock {\em SIAM J. Numer. Anal.}, 50(2):418--438, 2012.

\bibitem{arbogast1995implementation}
T.~Arbogast and Z.~Chen.
\newblock On the implementation of mixed methods as nonconforming methods for
  second-order elliptic problems.
\newblock {\em Math. Comp.}, 64(211):943--972, 1995.

\bibitem{babuska2011effect}
I.~Babu{\v{s}}ka, U.~Banerjee, and H.~Li.
\newblock The effect of numerical integration on the finite element
  approximation of linear functionals.
\newblock {\em Numer. Math.}, 117(1):65--88, 2011.

\bibitem{banerjee1989estimation}
U.~Banerjee and J.~E. Osborn.
\newblock Estimation of the effect of numerical integration in finite element
  eigenvalue approximation.
\newblock {\em Numer. Math.}, 56(8):735--762, 1989.

\bibitem{banerjee1992effect}
U.~Banerjee and M.~Suri.
\newblock The effect of numerical quadrature in the $p$-version of the finite
  element method.
\newblock {\em Math. Comp.}, 59(199):1--20, 1992.

\bibitem{cdssy}
Z.~Cai, J.~{Douglas,~Jr.}, J.~E. Santos, D.~Sheen, and X.~Ye.
\newblock Nonconforming quadrilateral finite elements: {A} correction.
\newblock {\em Calcolo}, 37(4):253--254, 2000.

\bibitem{cdy}
Z.~Cai, J.~{Douglas,~Jr.}, and X.~Ye.
\newblock A stable nonconforming quadrilateral finite element method for the
  stationary {S}tokes and {N}avier-{S}tokes equations.
\newblock {\em Calcolo}, 36:215--232, 1999.

\bibitem{ciarlet1991finite}
P.~Ciarlet.
\newblock Basic error estimates for elliptic problems.
\newblock In {\em Finite {Element} {Methods} ({Part} 1)}, volume~II of {\em
  Handbook of Numerical Analysis}, pages 17--351. North Holland, Amsterdam,
  1991.

\bibitem{ciarlet1972combined}
P.~G. Ciarlet and P.-A. Raviart.
\newblock The combined effect of curved boundaries and numerical integration in
  isoparametric finite element methods.
\newblock In {\em The mathematical foundations of the finite element method
  with applications to partial differential equations}, pages 409--474.
  Elsevier, 1972.

\bibitem{crouzeix-raviart}
M.~Crouzeix and P.~Raviart.
\newblock Conforming and nonconforming finite element methods for solving the
  stationary {Stokes} equations. {I.}
\newblock {\em R.A.I.R.O.-- Math. Model. Anal. Numer.}, 7(R-3):33--75, 1973.

\bibitem{dssy-nc-ell}
J.~{Douglas,~Jr.}, J.~E. Santos, D.~Sheen, and X.~Ye.
\newblock Nonconforming {G}alerkin methods based on quadrilateral elements for
  second order elliptic problems.
\newblock {\em ESAIM--Math. Model. Numer. Anal.}, 33(4):747--770, 1999.

\bibitem{fortin3d}
M.~Fortin.
\newblock A three-dimensional quadratic nonconforming element.
\newblock {\em Numer. Math.}, 46:269--279, 1985.

\bibitem{fortin-soulie-quad-nc-2d}
M.~Fortin and M.~Soulie.
\newblock A non-conforming piecewise quadratic finite element on the triangle.
\newblock {\em Int. J. Numer. Meth. Engrg.}, 19(4):505--520, 1983.

\bibitem{han84}
H.~Han.
\newblock Nonconforming elements in the mixed finite element method.
\newblock {\em J. Comp. Math.}, 2:223--233, 1984.

\bibitem{herbold1969effect}
R.~Herbold, M.~Schultz, and R.~Varga.
\newblock The effect of quadrature errors in the numerical solution of boundary
  value problems by variational techniques.
\newblock {\em aequationes mathematicae}, 3(3):247--270, 1969.

\bibitem{herbold1971effect}
R.~Herbold and R.~Varga.
\newblock The effect of quadrature errors in the numerical solution of
  two-dimensional boundary value problems by variational techniques.
\newblock {\em aequationes mathematicae}, 7(1):36--58, 1971.

\bibitem{jeon2013class}
Y.~Jeon, H.~Nam, D.~Sheen, and K.~Shim.
\newblock A class of nonparametric {DSSY} nonconforming quadrilateral elements.
\newblock {\em ESAIM: Mathematical Modelling and Numerical Analysis},
  47(6):1783--1796, 2013.

\bibitem{linke2016robust}
A.~Linke, G.~Matthies, and L.~Tobiska.
\newblock Robust arbitrary order mixed finite element methods for the
  incompressible {Stokes} equations with pressure independent velocity errors.
\newblock {\em ESAIM: Mathematical Modelling and Numerical Analysis},
  50(1):289--309, 2016.

\bibitem{meng2018new}
Z.~Meng, J.~Cui, and Z.~Luo.
\newblock A new rotated nonconforming quadrilateral element.
\newblock {\em Journal of Scientific Computing}, 74(1):324--335, 2018.

\bibitem{osborn1990estimation}
H.~E. Osborn and U.~Banerjee.
\newblock Estimation of the effect of numerical integration in finite element
  eigenvalue approximation.
\newblock {\em Numer. Math.}, 56:735--762, 1990.

\bibitem{park-sheen-p1quad}
C.~Park and D.~Sheen.
\newblock {$P_1$}-nonconforming quadrilateral finite element methods for
  second-order elliptic problems.
\newblock {\em SIAM J. Numer. Anal.}, 41(2):624--640, 2003.

\bibitem{rannacher-turek}
R.~Rannacher and S.~Turek.
\newblock Simple nonconforming quadrilateral {Stokes} element.
\newblock {\em Numer. Methods Partial Differential Equations}, 8:97--111, 1992.

\bibitem{stroud1966gaussian}
A.~H. Stroud and D.~Secrest.
\newblock {\em Gaussian Quadrature Formulas}.
\newblock Prentice--Hall, Englewood Cliffs, NJ, 1966.

\bibitem{turek}
S.~Turek.
\newblock {\em Efficient solvers for incompressible flow problems}, volume~6 of
  {\em Lecture Notes in Computational Science and Engineering}.
\newblock Springer, Berlin, 1999.

\bibitem{zeng2020optimal}
H.~Zeng, C.-S. Zhang, and S.~Zhang.
\newblock Optimal quadratic element on rectangular grids for {$H^1$} problems.
\newblock {\em BIT Numerical Mathematics}, pages 1--25, 2020.

\bibitem{zhou2016new}
X.~Zhou, Z.~Meng, and Z.~Luo.
\newblock New nonconforming finite elements on arbitrary convex quadrilateral
  meshes.
\newblock {\em Journal of Computational and Applied Mathematics}, 296:798--814,
  2016.

\end{thebibliography}

\end{document}